\documentclass{amsart}
\usepackage{verbatim}
\usepackage{rotating} %for \includegraphix[angle=270]
\usepackage{amssymb}
\usepackage{mathrsfs,amsfonts,amsmath,amssymb,epsfig,amscd,xy,amsthm,pb-diagram} 
\newtheorem{theorem}{Theorem}[section]
\newtheorem{proposition}[theorem]{Proposition}

\newtheorem{lemma}[theorem]{Lemma}

\newtheorem{corollary}[theorem]{Corollary}
\newtheorem{question}[theorem]{Question}
\newtheorem{definition}[theorem]{Definition}

\theoremstyle{plain}
\theoremstyle{remark}

\newtheorem{remark}[theorem]{Remark}
\newtheorem{example}[theorem]{Example}

\newcommand{\C}{{\mathbb C}}

\newcommand{\Q}{{\mathbb Q}}

\newcommand{\R}{{\mathbb R}}

\newcommand{\Z}{{\mathbb Z}}
\newcommand{\N}{{\mathbb N}}

\newcommand{\cC}{{\mathcal C}}
\newcommand{\cA}{{\mathcal A}}

\newcommand{\Qbar}{\bar{\Q}}

\DeclareMathOperator{\lcm}{lcm}

\DeclareMathOperator{\Norm}{N}

\newcommand{\bfs}{{\mathbf s}}

\newcommand{\cB}{\mathcal{B}}
\newcommand{\cF}{\mathcal{F}}

\newcommand{\cG}{\mathcal{G}}

\newcommand{\scrT}{\mathscr{T}}

\author{F\'elix Baril Boudreau}
\address{
F\'elix Baril Boudreau\\
Department of Mathematics and Computer Science\\
University of Lethbridge\\
4401 University Drive West
Lethbridge, T1K 3M4, Alberta, Canada
}
\email{felix.barilboudreau@uleth.ca}

\author{Erik Holmes}
\address{
Erik Holmes\\
Department of Mathematics and Statistics\\
University of Calgary\\
2500 University Drive NW\\
Calgary, T2N 1N4, Alberta, Canada
}
\email{erik.holmes@ucalgary.ca}

\author{Khoa D.~Nguyen}
\address{
Khoa D.~Nguyen \\
Department of Mathematics and Statistics\\
University of Calgary\\
2500 University Drive NW\\
Calgary, T2N 1N4, Alberta, Canada
}
\email{dangkhoa.nguyen@ucalgary.ca}

\thanks{F.~Baril Boudreau is partially supported by a PIMS postdoctoral fellowship at the University of Lethbridge. All the authors are partially supported by NSERC grant RGPIN-2018-03770 and CRC tier-2 research stipend 950-231716. They wish to thank Jason Bell and Tom Ward for helpful comments that improve the paper.}
\keywords{Artin-Mazur zeta function, Byszewski-Cornelissen conjecture, P\'olya-Carlson dichotomy}
\subjclass[2020]{Primary: 11J25, 13F25. Secondary: 37P35.}

\begin{document}
	\title[Adelic perturbation of rational functions]{Adelic perturbation of rational functions and applications}
	
	\date{July 2023}
	
	\begin{abstract}
 Let $\sum a_nx^n\in\bar{\mathbb{Q}}[[x]]$ be the power series representation of a rational function  and let $f:\ \{0,1,\ldots\}\rightarrow \bar{\mathbb{Q}}$ be a so-called almost quasi-polynomial. Under a necessary stability condition, we prove that $\sum f(n)a_nx^n$ satisfies the P\'olya-Carlson dichotomy: it is either a rational function or it cannot be extended analytically to a strictly larger domain than its disk of convergence. This latter property is much stronger than being transcendental. The first application and  motivation of our result is the solution of a conjecture by Byszewski-Cornelissen. This gives a  complete understanding of
 the analytic continuation behavior of the Artin-Mazur zeta function associated to a dynamical system on an abelian variety. Further applications include the solution of a conjecture by Bell-Miles-Ward and a significant case of an open problem by Royals-Ward.    
	\end{abstract}
	
	\maketitle

	\section{Introduction}\label{sec:intro}
        This first section is devoted to the application of our main result.  First we recall aspects of Byszewski-Cornelissen's work \cite{BC18_DO} concerning the dynamical zeta function of an endomorphism of an abelian variety and state a theorem that resolves their open problem (see Question~\ref{q:BC} and Theorem~\ref{thm:affirmative}). This gives a complete understanding of the analytic continuation behavior of such zeta functions. We then discuss two more results settling open problems by Bell-Miles-Ward \cite{BMW14_TA} and Royals-Ward \cite[Appendix]{BC18_DO}. The precise statements of these results and their proofs are given in the final section of the paper.  We refer the reader to the next section for the statement of our main result and related discussion since they require several further definitions.

        Let $\varphi:\ X\rightarrow X$ be a map on a set $X$. For each integer $k\geq 1$, let $N_k(\varphi)$ denote the number of fixed points of the $k$-th fold iterate $\varphi^k:=\varphi\circ\ldots\circ \varphi$ ($k$ times). Assume that $N_k(\varphi)$ is finite for every $k$. Then we can 
        define the dynamical or Artin-Mazur zeta function:
        $$\zeta_{\varphi}(x)=\exp\left(\sum_{k=1}^{\infty}\frac{N_k(\varphi)}{k}x^k\right).$$
		The role of $\zeta_{\varphi}$ as an important invariant of a dynamical system was  first recognized by Artin-Mazur \cite{AM65_OP} perhaps with motivation from the Weil conjectures (when $X$ is a variety over a finite field and $\varphi$ is the Frobenius then $\zeta_{\varphi}$ is
		the zeta function of $X$). This was further advocated by Smale \cite{Sma67_DF} and many other authors. As mentioned in \cite[p.~84]{AM65_OP} and \cite[p.~764]{Sma67_DF}, a fundamental problem is to determine whether $\zeta_{\varphi}(x)$ is rational, algebraic, or transcendental. Thanks to earlier results by various authors, it has been noted by Bell, Miles, and Ward \cite{BMW14_TA} that for many dynamical systems, the corresponding zeta function satisfies the P\'olya-Carlson dichotomy: it is either rational or it cannot be extended analytically beyond the disk of convergence. The latter property is much stronger than being transcendental. We refer the reader to \cite{BMW14_TA,BGNS23_AG} and their references for further details. 

        For the rest of this section, let $\varphi$ be an endomorphism of an abelian variety $X$ defined over an algebraically closed field $K$ such that $\ker(\varphi^k-1)$ is finite for every $k\geq 1$. As explained in \cite[Proposition~5.1]{BC18_DO}, if $g=\dim(X)$ then we have
        algebraic numbers $\xi_1,\ldots,\xi_{2g}$ none of which is a root of unity
        such that
        $$\deg(\varphi^k-1)=\prod_{i=1}^{2g}(\xi_i^k-1)$$
        for every $k\geq 1$.  By \cite[Proposition~5.2]{BC18_DO}, the dynamical zeta 
        function $\zeta_{\varphi}(x)$ has radius of convergence $1/\Lambda$        
   		where     
        $$\Lambda:=\prod_{i=1}^{2g}\max\{\vert\xi_i\vert,1\}.$$
        
        When $\varphi^k-1$ is separable for every $k$ (for example, this happens when $K$ has characteristic $0$), we have 
        $$N_k(\varphi)=\deg(\varphi^k-1)$$
        and therefore $\zeta_{\varphi}(x)$ is a rational function. In order to complete our understanding of $\zeta_{\varphi}(x)$ in terms of the P\'olya-Carlson dichotomy, Byszewski and Cornelissen ask the following \cite[Question~5.6]{BC18_DO}:

        \begin{question}[Byszewski-Cornelissen]\label{q:BC}
            Suppose that $\varphi^k-1$ is not separable for some $k\geq 1$, does $\zeta_{\varphi}(x)$ admit the circle of radius $1/\Lambda$ as a natural boundary?
        \end{question}

        Although they pose this as a question, it seems evident that they conjecture an affirmative answer by establishing two supporting results. First, they prove the weaker property that $\zeta_{\varphi}(x)$ is not 
        D-finite under the assumption of Question~\ref{q:BC} \cite[Theorem~4.3]{BC18_DO}. Second, if we assume the extra ``unique dominant root'' condition, equivalently $\vert \xi_i\vert\neq 1$ for $1\leq i\leq 2g$ \cite[Proposition~5.1]{BC18_DO}, then Question~\ref{q:BC} has an affirmative answer \cite[Theorem~5.5]{BC18_DO}. We obtain the following unconditionally as a consequence of our main result:

        \begin{theorem}\label{thm:affirmative}
            Question~\ref{q:BC} has an affirmative answer.
        \end{theorem}

        The proof of \cite[Theorem~5.5]{BC18_DO} relies on a theorem of Royals-Ward \cite[Appendix]{BC18_DO} which considers the power series
        $\displaystyle\sum a_nx^n\in\Q[[x]]$ of a rational function and its ``adelic perturbation'' 
        $\displaystyle\sum \vert a_n\vert_S a_n x^n$ where $S$ is a set of 
        non-archimedean places of $\Q$ and $\vert m\vert_S:=\prod_{\ell\in S}\vert m\vert_{\ell}$ for every $m\in \Q$. These power series are important since they are closely related to the dynamical zeta functions
        on compact abelian groups (see \cite{CEW97,Mil08_PP,BMW14_TA}). 
        Although the dynamics on such groups has been studied for more than a hundred years 
        with profound applications in combinatorics and number theory, basic questions such as 
        \cite[p.~653]{BMW14_TA} or the more notorious Furstenberg's $\times 2\times 3$-Conjecture \cite{Fur67}
        remain open. 
        Royals and Ward treat the case when
        $\vert S\vert<\infty$ and the $a_n$'s have a special form. They also state an open question concerning the P\'olya-Carlson dichotomy for such perturbation \cite[p.~2228]{BC18_DO}. 
        
        The second  consequence of our main result settles this question in the case $\vert S\vert<\infty$ under a  mild stability condition (see Corollary~\ref{cor:RW} and Remark~\ref{rem:RW}). Finally, the third consequence of our main result resolves an open problem by Bell-Miles-Ward \cite[p.~664]{BMW14_TA} concerning the dynamical zeta functions of automorphisms on compact abelian groups. As in \cite{BC18_DO}, the authors of \cite{BMW14_TA} could establish the desired dichotomy under an extra unique dominant root condition. They remark that although this condition is essential in their proof they still conjecture that the same result holds without the condition (see Theorem~\ref{thm:BMW} and Remark~\ref{rem:BMW}). With these various applications, it is evident that our results (Theorem~\ref{thm:main} and Theorem~\ref{thm:SML variant}) provide a systematic way to successfully attack the P\'olya-Carlson dichotomy problem for
        the above adelic perturbation $\displaystyle\sum \vert a_n\vert_S a_n x^n$ when $\vert S\vert<\infty$. To the best of our knowledge, there is no general approach to the problem when $S$ is infinite and this will be the subject of future work.

        We conclude this section with a brief discussion about the motivation and techniques used in this paper. Inspired by the results and techniques in \cite{BNZ20_DF,BNZ23_DF2}, Bell, Gunn, Nguyen, and Saunders prove a general criterion for the P\'olya-Carlson dichotomy \cite[Theorem~1.2]{BGNS23_AG}. This has new applications to the Artin-Mazur zeta functions of endomorphisms of positive characteristic tori in \cite[Theorem~4.2]{BNZ23_DF2} and of hyperbolic FAD-systems in \cite[Theorem~C]{BCH}. During the review process of \cite{BGNS23_AG}, a reviewer asked whether one could apply its results or methods to settle Question~\ref{q:BC}. Indeed we had been trying to do so even before the reviewer's suggestion. After several failed attempts, we realize that a different approach is needed in order to take advantage of the special feature that the original series
        $\sum a_n x^n$ is rational. Moreover, under a mild stability condition, the perturbation $\vert a_n\vert_S$ itself is essentially an example 
        of an almost quasi-polynomial. 
        This large family of almost quasi-polynomials allows us to extend the relevant power series to
		a full disk albeit with poles contributed by elements in $\Qbar(x)$ when the series
		does not admit the circle of the radius of convergence as a natural boundary (see Step 1 and Step 2 in the proof of Theorem~\ref{thm:main}).	        
        The notions of stability and almost quasi-polynomials, the main result, and other definitions will be given in the next section. The final ingredient is a very useful diophantine inequality (see Proposition~\ref{prop:KMN}) that is a consequence of the Schmidt's Subspace Theorem. This is used in not only the proof of the main theorem but also the proof that certain functions belong to the set of almost quasi-polynomials.

        \section{Perturbation by an almost quasi-polynomial}
        From now on, let $\N$ be the set of positive integers and let $\N_0=\N\cup\{0\}$. We fix an embedding of $\Qbar$ into $\C$ and use $\vert\cdot\vert$ to denote the usual absolute value on $\C$ as well as its restriction to $\Qbar$.

        For a number field $K$,
	let $M_K^\infty$ denote the set of archimedean places of $K$ and let 
	$M_K^0$ denote the set of finite places.  
  	We write $M_K=M_K^\infty\cup M_K^0$. For every place
  	$w\in M_K$, let $K_w$ denote the completion of 
  	$K$ with respect to $w$ and we let
  	$\delta_w=[K_w:\Q_v]$ where $v$ is the restriction of
  	$w$ to $\Q$. Following \cite[Chapter~1]{Voj87}, we normalize $\vert\cdot\vert_w$ on $K$ as follows. We always take
 $\vert\cdot \vert_v$ on $\Q$ to be either the usual Euclidean absolute value or the $p$-adic absolute value for some prime $p$. Then we let $\vert\cdot\vert_w$ be the absolute value on $K$ extending $\vert\cdot\vert_v$
 defined by
 $$\vert a\vert_w=\vert \Norm_{K_w/\Q_v}(a)\vert_v^{1/\delta_w}\quad \text{for $a\in K$}.$$

 \begin{remark}\label{rem:normalization}
     Our normalization of $\vert\cdot\vert_w$ is different from the one in   \cite{Mil08_PP,BMW14_TA} which uses the formula $\vert\Norm_{K_w/\Q_v}(a)\vert_v$.
     Both of these are different from the one in \cite{BG06_HI,KMN19_AA,BNZ20_DF}
     which uses the formula $\vert\Norm_{K_w/\Q_v}(a)\vert_v^{1/[K:\Q]}$. One of 
     the reasons for our choice is that we will use a Skolem-Mahler-Lech type argument to work with an analytic function $\Z_p\rightarrow K_w$ where $w$ restricts to $p\in M_\Q^0$ and it is more convenient to
     have $\vert\cdot\vert_w$ extend $\vert\cdot\vert_p$. 
 \end{remark}

We define the absolute logarithmic Weil height $h:\ \Qbar\rightarrow\R_{\geq 0}$. Let $a\in\Qbar$, choose a number field $K$ containing $a$, and define:
$$h(a)=\frac{1}{[K:\Q]}\sum_{w\in M_K}\log\max\{\vert a\vert_w^{\delta_w},1\}.$$
This is independent of the choice of $K$ \cite[Chapter~1]{Voj87}. For basic properties of $h$, see \cite{Voj87,BG06_HI,BNZ20_DF}. 

\begin{definition}\label{def:pol-exp seq}
    \begin{itemize}
        \item  [(i)] A sequence $(u_n)_{n\geq 0}$ of algebraic numbers is said to be a linear recurrence or polynomial-exponential sequence if there exist $r\in\N$, distinct non-zero $\alpha_1,\ldots,\alpha_r\in \Qbar$, and non-zero $P_1(x),\ldots,P_r(x)\in \Qbar[x]$ such that
    \begin{equation}\label{eq:def u_n}
    u_n=P_1(n)\alpha_1^n+\cdots+P_r(n)\alpha_r^n\quad\text{for every $n$.}
    \end{equation}
    The $\alpha_i$'s are called the characteristic roots of the given sequence. We say that $(u_n)_{n\geq 0}$ is defined over a number field $K$ if the $\alpha_i$'s and the coefficients of the $P_i$'s are in $K$.

        \item [(ii)] Let $r\in\N$. A tuple of non-zero algebraic numbers $(a_1,\ldots,a_r)$ is said to be non-degenerate if $a_i/a_j$ is not a root of unity for $1\leq i\neq j\leq r$.

        \item [(iii)] A linear recurrence sequence is called non-degenerate if its tuple of characteristic roots is non-degenerate.
    \end{itemize}
\end{definition}

\begin{remark}
    Definition~\ref{def:pol-exp seq} does not include the zero sequence. One can do so by allowing $r=0$ and the various statements will be either vacuously true or obviously false. It is simply our choice to ignore this trivial case.  
\end{remark}

\begin{definition}\label{def:v-stable}
    Let $(u_n)_{n\geq 0}$ be a linear recurrence sequence defined over a number field $K$ as in \eqref{eq:def u_n}. Let $v\in M_K$.

    Let $M=\max\{\vert \alpha_i\vert_v:\ 1\leq i\leq r\}$ and let $I=\{1\leq i\leq r:\ \vert\alpha_i\vert_v=M\}$. By the essential $v$-part of $(u_n)_{n\geq 0}$, we mean the sequence
    $$u_{v,n}:=\sum_{i\in I}P_i(n)\alpha_i^n.$$
    The sequence $(u_n)_{n\geq 0}$ is called $v$-stable if for every $a,b\in\N$ the
    sequence $(u_{v,an+b})_{n\geq 0}$ is not the zero sequence. The sequence $(u_n)_{n\geq 0}$ is called stable if it is $\vert\cdot\vert$-stable where, as always, $\vert\cdot\vert$ is the restriction of the usual absolute value of $\C$ on $K$.
\end{definition}

\begin{example}
    Let $\zeta=\exp(2\pi i/3)$, $K=\Q(\zeta)$, and
    $$u_n=(1/2)^n+(-1/2)^n+5^n+2\cdot(5\zeta)^n+3\cdot (5\zeta^2)^n.$$

    Let $v$ be a place of $K$ lying above the $2$-adic place. We have
    $u_{v,2n}=2\cdot (1/2)^{2n}$ and $u_{v,2n+1}=0$ for every $n$. Hence the sequence is not $v$-stable.

    To check whether it is stable, it suffices to consider the congruence classes mod $3$:
    $$u_{\vert\cdot\vert,3n}= 6\cdot 5^{3n},\quad u_{\vert\cdot\vert,3n+1}=(-\zeta-2)\cdot 5^{3n+1},\quad u_{\vert\cdot\vert,3n+2}=(\zeta-1)\cdot 5^{3n+2}.$$
    Therefore the given sequence is stable.
\end{example}

\begin{example}\label{eg:non-degenerate is stable}
    A non-degenerate linear recurrence sequence is $v$-stable for every $v$. Although it is not true in general that the essential $v$-part of a product is the product of the essential $v$-parts (for example, consider the sequences $2^n+(-2)^n+1$
    and $2^n-(-2)^n$ with $v=\vert\cdot\vert$), 
    it is true under the further condition that the involving sequences are $v$-stable. 
    Indeed, suppose $(u_n)_{n\geq 0}$ and $(t_n)_{n\geq 0}$ are $v$-stable sequences. 
    By the Skolem-Mahler-Lech theorem \cite{Sko34,Mah35,Lec53}, the $v$-stability condition is equivalent to
    $u_{v,n}t_{v,n}\neq 0$ for all sufficiently large $n$. Then it follows that the essential $v$-part of the sequence
    $(u_nt_n)_{n\geq 0}$ is $(u_{v,n}t_{v,n})_{n\geq 0}$ and therefore the sequence $(u_nv_n)_{n\geq 0}$ is $v$-stable. In particular, sequences of the form
    $$u_n=\prod_{i=1}^{m}(\xi_i^n-1)$$
    where $m\in\N$ and $\xi_1,\ldots,\xi_m$ are (not necessarily distinct) algebraic numbers none of which is a root of unity are $v$-stable for every $v$.
\end{example}

\begin{remark}
    We refer the reader to Corollary~\ref{cor:characterization} for another characterization of $v$-stable sequences.
\end{remark}

\begin{definition}
    Let $\sum a_nx^n\in\Qbar[[x]]$ be the power series representation of a rational function that is not a polynomial. Let $(u_n)_{n\geq 0}$ be the linear recurrence sequence defined over a number field $K$ such that $a_n=u_n$ for all sufficiently large $n$. Let $v\in M_K$. We say that the given power series is $v$-stable (respectively stable) if $(u_n)_{n\geq 0}$ is $v$-stable (respectively stable).  
\end{definition}

The final ingredient for our main result is the notion of an almost quasi-polynomial. While its definition is somewhat technical, it appears to be the right tool in solving the P\'olya-Carlson dichotomy problem for adelic perturbation of rational functions proposed by Royals-Ward \cite[Appendix]{BC18_DO}, at least when the set of places $S$ is finite. First, we start with the following:
\begin{definition}
    A function $f:\ \N_0\rightarrow \Qbar$ is said to be a quasi-polynomial if there exist $d\in\N$ and polynomials $P_0,\ldots,P_{d-1}\in\Qbar[x]$
    such that for every sufficiently large $n$, we have:
    $$f(n)=P_{j}(n)\quad \text{if $n\equiv j\bmod d$}.$$
    Moreover, if the $P_{j}$'s are constant polynomials, we call $f$ a quasi-constant.
\end{definition}

\begin{remark}
    This is slightly more general than the definition in \cite[Chapter~4.4]{Sta12_EC1} since we require that $f(n)$ be given by the $P_j$'s only when $n$ is sufficiently large.
\end{remark}

\begin{definition}\label{def:almost q-p}
    A function $f:\ \N_0\rightarrow \Qbar$ is said to be an almost quasi-polynomial if the following three properties hold:
    \begin{itemize}
        \item [(P1)] There exists a number field $K$ such that $f(n)\in K$ for every $n$.
        \item [(P2)] $h(f(n))=o(n)$ as $n\to\infty$.
        \item [(P3)] There exist an increasing sequence $d_1<d_2<\ldots$ of positive integers, a ``good'' set of congruence classes $\cG_i\subseteq\{0,\ldots,d_i-1\}$ for each $i\geq 1$, a polynomial $P_{i,j}(x)\in\Qbar[x]$ for each $i\geq 1$ and $j\in\cG_i$ such that:
            \begin{itemize}
                \item [(P3a)] $f(n)=P_{i,j}(n)$ for all sufficiently large $n$ such that  $n\equiv j\bmod d_i$ and
                \item [(P3b)] $\displaystyle\lim_{i\to\infty} \frac{\vert \cG_i\vert}{d_i}=1$.
            \end{itemize}
    \end{itemize}
	Moreover, if the $P_{i,j}$'s in (P3) are constant polynomials, we say that $f$ is an almost quasi-constant.
\end{definition}

\begin{remark}
    There is an extra flexibility in (P3a): the ``sufficient largeness'' condition on $n$ need not be uniform in $i$. More explicitly, we have $C_i\geq 0$ for each $i\geq 1$ such that $f(n)=P_{i,j}(n)$ for $n\geq C_i$ and $n\equiv j\bmod d_i$; there is no further requirement (for example, uniform boundedness) on the $C_i$'s as $i$ varies. Property (P2)
    is a growth condition on $h(f(n))$. The motivation for our terminology comes from (P3): as $i\to\infty$, for \emph{almost all} congruence classes modulo $d_i$, the function $f(n)$ when $n$ is sufficiently large is given by a polynomial on each congruence class.   
\end{remark}

\begin{example}\label{eg:|n|_p}
    Obviously every quasi-polynomial (respectively quasi-constant) is an almost quasi-polynomial (respectively almost quasi-constant). Perhaps the simplest example of an almost quasi-constant that is not a quasi-polynomial is $f(n)=\vert n\vert_p$ where $p$ is a given prime number. We may take $d_i=p^i$ and $\cG_i=\{1,\ldots,d_i-1\}$ (i.e. the non-zero congruence classes modulo $d_i$) then use the fact that $f(n)$ is constant on each congruence class $j\bmod d_i$ for $j\in\cG_i$. We refer the reader to Remark~\ref{rem:necessary} and Theorem~\ref{thm:SML variant}
    for many more examples. We leave it as an exercise to prove that sums and products of almost quasi-polynomials (respectively almost quasi-constants) are still almost quasi-polynomials (respectively almost quasi-constants). Hence the almost quasi-constants form a $\Qbar$-subalgebra of the almost quasi-polynomials.
    While the set of quasi-polynomials is countable, Remark~\ref{rem:necessary} 
    gives uncountably many almost quasi-constants. Therefore the algebra of almost quasi-constants cannot be generated by the countably many functions of the form $\vert n\vert_p$ and, more generally, by the countably many functions in Theorem~\ref{thm:SML variant}.
\end{example}

Our main result is the following:
\begin{theorem}\label{thm:main}
    Let $\displaystyle\sum_{n=0}^{\infty} a_nx^n\in\Qbar[[x]]$ be the power series representation of a rational function that is not a polynomial. Suppose that this series is stable. Let $R\in (0,\infty)$ be the radius of convergence. 
    Let $f:\ \N_0\rightarrow\Qbar$ be an almost quasi-polynomial. Then the power series $\displaystyle\sum_{n=0}^{\infty} f(n)a_nx^n$ either
    \begin{itemize}
        \item [(i)] admits the circle of radius $R$ as a natural boundary, or
        \item [(ii)] represents a rational function.
    \end{itemize}
    Moreover, (ii) happens if and only if $f$ is a quasi-polynomial.
\end{theorem}

\begin{remark}\label{rem:radius}
    When $f(n)=0$ for all sufficiently large $n$, case (ii) is obvious since
    $\displaystyle\sum_{n=0}^{\infty} f(n)a_nx^n$ is a polynomial. Otherwise, the proof of Theorem~\ref{thm:main} shows that $\displaystyle\sum_{n=0}^{\infty} f(n)a_nx^n$ has radius of convergence $R$.
\end{remark}

\begin{remark}\label{rem:necessary}
    We explain why the stability assumption is necessary for the conclusion of Theorem~\ref{thm:main}. Consider $a_n=2^n+(-2)^n+1$ which is not stable. Let $\bfs=(s_k)_{k\geq 1}$ be a sequence of non-zero integers such that:
    $$s_1=1\quad \text{and}\quad \lim_{k\to\infty}\frac{\log\vert s_k\vert}{2^k}=0.$$
    Consider $f_{\bfs}:\ \N_0\rightarrow \Qbar$ defined by the following (infinitely many) cases:
    \begin{itemize}
        \item $f(n)=s_1$ if $n\equiv 0\bmod 2$,
        \item $f(n)=s_2$ if $n\equiv 1\bmod 4$,
        \item $f(n)=s_3$ if $n\equiv 3\bmod 8,\ldots$
        \item $f(n)=s_k$ if $n\equiv 2^{k-1}-1\bmod 2^k,\ldots$
    \end{itemize}
    The first $k$ lines define $f_{\bfs}(n)$ for $n$ belonging to all except the congruence class of $-1$ modulo $2^k$. These also constitute all except the congruence classes of $2^{k}-1$ and $-1$ modulo $2^{k+1}$. Then the $(k+1)$-th line defines $f_{\bfs}(n)$ for $n\equiv 2^{k}-1\bmod 2^{k+1}$.  Therefore $f_{\bfs}$ is well-defined on $\N_0$ (even on the bigger domain $\Z\setminus\{-1\}$). We now prove that $f_{\bfs}$ is an almost quasi-constant. Properties (P1) and (P3) hold: we may take $d_i=2^i$, then $f_{\bfs}$ is constant on each congruence class modulo $d_i$ except the one given by $-1\bmod d_i$. To verify (P2), let $n\in\N_0$, then let $k$ be the smallest positive integer such that $n<2^k-1$ so that $f_{\bfs}(n)$ is defined within the first $k$ lines. Therefore:
    $$h(f_{\bfs}(n))=\log \vert f_{\bfs}(n)\vert\leq \max_{1\leq i\leq k} \log\vert s_i\vert=o(2^k)=o(n).$$
    We have:
    \begin{align*}
    \sum_{n=0}^{\infty}f_{\bfs}(n)a_nx^n&=\sum_{\text{even}\ n} (2\cdot 2^n+1)x^n+\sum_{\text{odd}\ n} f_{\bfs}(n)x^n\\
    &=\frac{2}{1-4x^2}+\frac{1}{1-x^2}+\sum_{\text{odd}\ n} f_{\bfs}(n)x^n
    \end{align*}
    has radius of convergence $1/2$ while $\displaystyle \sum_{\text{odd}\ n} f_{\bfs}(n)x^n$ has radius of convergence $1$ (since $1\leq \vert f_{\bfs}(n)\vert=e^{o(n)}$). Therefore $\displaystyle\sum_{n=0}^{\infty}f_{\bfs}(n)a_nx^n$ does not admit the circle of radius $1/2$ as a natural boundary.
    On the other hand, there are uncountably many choices for $\bfs$ while $\displaystyle\sum_{n=0}^{\infty}f_{\bfs}(n)a_nx^n$
    represents a rational function for only countably many $\bfs$. Hence the conclusion of Theorem~\ref{thm:main} is false for uncountably many almost quasi-constants $f_{\bfs}$.
\end{remark}

    \section{Proof of Theorem~\ref{thm:main} and Theorem~\ref{thm:affirmative}}
    The following application of the Subspace Theorem \cite{Sch70_SA,ES02_AQ} will be used several times. 
    \begin{proposition}\label{prop:KMN}
        Let $\ell:\ \N\rightarrow\R_{\geq 0}$ be a sublinear function meaning
        $\displaystyle\lim_{n\to\infty}\frac{\ell(n)}{n}=0$. Let $r\in\N$, let $K$ be a number field, and let $(\alpha_1,\ldots,\alpha_r)$ be a non-degenerate tuple of elements of $K^*$. Let $v\in M_K$ and let $0<B<\displaystyle\max_{1\leq i\leq r}\vert \alpha_i\vert_v$.
        Then there are only finitely many tuples $(n,b_1,\ldots,b_r)\in \N\times (K^*)^r$ such that:
        \begin{equation}\label{eq:KMN}
        \vert b_1\alpha_1^n+\cdots+b_r\alpha_r^n\vert_v\leq B^n\quad\text{and}\quad\max_{1\leq i\leq r} h(b_i)\leq \ell(n).
        \end{equation}
    \end{proposition}
    \begin{proof}
        This is from \cite[Section~2]{KMN19_AA} following earlier work of Evertse \cite{Eve84_OS} and Corvaja-Zannier \cite{CZ04_OT}. By \cite[Proposition~2.3]{KMN19_AA}, each $(n,b_1,\ldots,b_r)$ where \eqref{eq:KMN} holds satisfies at least one of finitely many non-trivial relations of the form:
        $$c_1b_1\alpha_1^n+\cdots+c_rb_r\alpha_r^n=0.$$
        Then we apply \cite[Proposition~2.2]{KMN19_AA} to get the desired conclusion.
    \end{proof}

    \begin{corollary}\label{cor:characterization}
        Let $(u_n)_{n\geq 0}$ be a linear recurrence sequence defined over a number field $K$ with
        $$u_n=P_1(n)\alpha_1^n+\cdots+P_r(n)\alpha_r^n$$
        as in \eqref{eq:def u_n}. Let $v\in M_K$, $M=\displaystyle\max_{1\leq i\leq r} \vert\alpha_i\vert_v$, and the essential $v$-part $(u_{v,n})_{n\geq 0}$
        be as in Definition~\ref{def:v-stable}. The following are equivalent:
        \begin{itemize}
            \item [(i)] The sequence $(u_n)_{n\geq 0}$ is $v$-stable.
            \item [(ii)] For every $B\in (0,M)$, we have $\vert u_n\vert_v\geq B^n$
            for all sufficiently large $n$.
            \item [(iii)] For every $B\in (0,M)$, we have
            $\vert u_{v,n}\vert_v\geq B^n$ for all sufficiently large $n$.
        \end{itemize}
    \end{corollary}
    \begin{proof}
        The equivalence (ii)$\Longleftrightarrow$(iii) and the implication
        (iii)$\Longrightarrow$(i) follow immediately from Definition~\ref{def:v-stable}. It remains to prove (i)$\Longrightarrow$(iii). 
        
        Let $L$ be the lcm of the orders of the roots of unity
        among the $\alpha_i/\alpha_j$.
        The $v$-stability of
        $(u_n)_{n\geq 0}$ imply that for each $b\in\{1,\ldots,L\}$, the sequence
        $u_{v,Ln+b}$ has the form:
        $$u_{v,Ln+b}=Q_1(n)\beta_1^n+\cdots+Q_s(n)\beta_s^n$$
        where $s\in\N$, $(\beta_1,\ldots,\beta_s)$ is a non-degenerate tuple 
        of non-zero algebraic numbers with $\vert\beta_1\vert_v=\ldots=\vert\beta_s\vert_v=M^L$, and $Q_i(x)\in\Qbar[x]\setminus\{0\}$. Let $B^L<B_1<M^L$, then Proposition~\ref{prop:KMN} gives that
        $$\vert u_{v,Ln+b}\vert_v> B_1^n > B^{Ln+b}$$
        for all sufficiently large $n$. Hence (iii) holds.
        \end{proof}

    \begin{proof}[Proof of Theorem~\ref{thm:main}]
        We assume the notation of Theorem~\ref{thm:main}. Put $F(x)=\displaystyle\sum_{n=0}^{\infty}f(n)a_nx^n$.  We can express
        $$a_n=P_1(n)\alpha_1^n+\cdots+P_r(n)\alpha_r^n$$
        with $r>0$, distinct non-zero $\alpha_1,\ldots,\alpha_r\in\Qbar$, and non-zero
        $P_1,\ldots,P_r\in\Qbar[x]$ for all large $n$. Without loss of generality, we have $1\leq s\leq r$ and 
        $$\vert\alpha_1\vert=\ldots=\vert \alpha_s\vert=\frac{1}{R}=\max_{1\leq i\leq r} \vert\alpha_i\vert$$
        while $\vert\alpha_i\vert<1/R$ for $s+1\leq i\leq r$.

        Let $K$ be a number field containing the $\alpha_i$'s, coefficients of the $P_i$'s, and the $f(n)$'s. If $f(n)=0$ for all large $n$ then we are done. We assume otherwise from now on. Let $\epsilon>0$. Corollary~\ref{cor:characterization} gives 
        $$(1/R)^{(1-\epsilon)n}<\vert a_n\vert<(1/R)^{(1+\epsilon)n}$$
        for all large $n$. From $h(f(n))=o(n)$, we have
$\vert f(n)\vert<e^{\epsilon n}$ and if $f(n)\neq 0$ (which happens for infinitely many $n$ under the current assumption) then we also have $e^{-\epsilon n}<\vert f(n)\vert$ when $n$ is large. Letting $\epsilon\to 0$, we have that the radius of convergence of $F(x)$ is $R$.

        Let  $d_1<d_2<\ldots$ be a sequence of positive integers, let $\cG_i\subseteq\{0,\ldots,d_i-1\}$ for each $i\geq 1$, and let $P_{i,j}\in\Qbar[x]$ for each $i\geq 1$ and $j\in \cG_i$ satisfy Property (P3) in Definition~\ref{def:almost q-p}.
        For $i\geq 1$, let $\cB_i:=\{0,\ldots,d_i-1\}\setminus\cG_i$ be the set of
        ``bad'' congruence classes modulo $d_i$ on which $f(n)$ might not be given by polynomial functions. If $\cB_i=\emptyset$, equivalently $\cG_i=\{0,\ldots,d_i-1\}$, then $f$ is a quasi-polynomial and 
        $F(x)$ represents a rational function. 
        Therefore we may assume that $\cB_i\neq \emptyset$ for every $i$.
        Put $b(i)=\vert \cB_i\vert\geq 1$ for $i\geq 1$, we have $b(i)=o(d_i)$ thanks to Property (P3b).

        \textbf{Step 1:} fix an arbitrary $i\geq 1$.
        Write $\cB_i=\{m_1,\ldots,m_{b(i)}\}$. Let $\zeta_1,\ldots,\zeta_{b(i)}$ be (not necessarily distinct) $d_i$-th roots of unity. For any subset $J$ of 
        $\{1,\ldots,b(i)\}$, we simplify the notation by using
        $$\zeta_J:=\prod_{j\in J}\zeta_j,\quad\zeta_J^{m_J}:=\prod_{j\in J}\zeta_j^{m_j},\quad\text{and}\quad \zeta_J^{-m_J}:=\prod_{j\in J}\zeta_j^{-m_j};$$
        when $J=\emptyset$ we interpret all the above as $1$.  \emph{The goal of this step} is to exhibit a relation among power series of the form $F(\zeta_Jx)$ where $J$ runs over the $2^{b(i)}$ many subsets of
        $\{1,\ldots,b(i)\}$.

        We consider the auxiliary power series
        \begin{equation}\label{eq:cA(x)}
            \cA(x):=\sum_{n=0}^{\infty} (1-\zeta_1^{-m_1}\zeta_1^n)\cdots (1-\zeta_{b(i)}^{-m_{b(i)}}\zeta_{b(i)}^n)f(n)a_nx^n.
        \end{equation}
        Obviously, the factor $(1-\zeta_1^{-m_1}\zeta_1^n)\cdots (1-\zeta_{b(i)}^{-m_{b(i)}}\zeta_{b(i)}^n)=0$ if $n\equiv m_j\bmod d_i$ for some $1\leq j\leq b(i)$. Therefore it suffices to consider $n$ in the good congruence classes 
        $\cG_i$ modulo $d_i$:
        $$\cA(x)=\sum_{j\in\cG_i}\sum_{n\in j+d_i\N_0} (1-\zeta_1^{-m_1+j})\cdots(1-\zeta_{b(i)}^{-m_{b(i)}+j})f(n)a_nx^n.$$
        Since $\sum a_nx^n$ is rational and $f(n)=P_{i,j}(n)$ for every sufficiently large $n\in j+d_i\N_0$, each series
        $$\sum_{n\in j+d_i\N_0} (1-\zeta_1^{-m_1+j})\cdots(1-\zeta_{b(i)}^{-m_{b(i)}+j})f(n)a_nx^n$$
        is rational. Therefore $\cA(x)$ is rational.
        
        On the other hand, we can expand \eqref{eq:cA(x)} directly:
        \begin{align*}
        \cA(x)&=\sum_{J\subseteq \{1,\ldots,b(i)\}}\sum_{n=0}^{\infty}(-1)^{\vert J\vert}\cdot \zeta_J^{-m_J}\cdot \zeta_J^n\cdot f(n)a_nx^n\\
        &=\sum_{J\subseteq \{1,\ldots,b(i)\}}(-1)^{\vert J\vert}\cdot \zeta_J^{-m_J}\cdot F(\zeta_Jx).
        \end{align*}
    The conclusion of this step is that there exist non-zero algebraic numbers $c_J$ depending on $\cB_i$, $J$, and the given $\zeta_1,\ldots,\zeta_{b(i)}$ such
    that
    $$\sum_{J\subseteq\{1,\ldots,b(i)\}} c_JF(\zeta_Jx)\in\Qbar(x).$$

    \emph{From now on,} we assume that $F(x)$ does not admit the circle of radius $R$ as a natural boundary. We will prove that $f$ is a quasi-polynomial and hence $F$ represents a rational function.

    \textbf{Step 2:} for $a\in\C$ and $\rho>0$, let $D(a,\rho)$ denote the open disk of radius $\rho$ centered at $a$. Our assumption on $F(x)$ means 
    there is a point $z=Re^{i\theta}$ with $\theta\in [0,2\pi)$ and an open set $U$ containing $z$ such that
    $F$ extends to an analytic function on $D(0,R)\cup U$. For an open set $V$ of $\C$, let 
    $$\cF(V):=\left\{\text{meromorphic}\ g:\ V\rightarrow \C\cup\{\infty\}:\ \begin{array}{l}
        \text{there exists $g_1\in\Qbar(x)$ such that}\\
        \text{$g-g_1$ is analytic on $V$}
    \end{array}\right\}$$
%    $\cF(V)$ denote the $\Qbar$-algebra of meromorphic functions $g:\ V\rightarrow\C\cup\{\infty\}$ such that $g$ has only finitely many poles and these poles belong to $\Qbar$. 
    \emph{The goal of this step} is to extend $F$ to a function $g\in \cF(D(0,R'))$ with $R'>R$.

    For $T\subseteq\R$ and $\rho>0$, let $E(T,\rho):=\{\rho e^{it}:\ t\in T\}$ denote the corresponding arc on the circle of radius $\rho$.    Define
    $$\scrT=\left\{t\in [\theta,\infty):\ \begin{array}{l}\text{there exists an open set $V$ containing $E([\theta,t],R)$ such that}\\ 
     \text{$F$ extends to $g\in\cF(D(0,R)\cup V)$.}
    \end{array}\right\}$$

    By our assumption on $z=Re^{i\theta}$ and $U$, the set $\scrT$ contains
    an interval $[\theta,\theta']$ with $\theta'>\theta$. If $\scrT$ contains some $t$ with $t-\theta\geq 2\pi$ then there exists an open set $V$ containing the circle $E([\theta,t],R)$ such that $F$ extends to
    an element in $\cF(D(0,R)\cup V)$. And since $D(0,R)\cup V$ contains some
    $D(0,R')$ with $R'>R$, we obtain our goal.

    For the remainder of this step, we assume $t<\theta+2\pi$ for every $t\in \scrT$ and arrive at a contradiction. Let $t^*:=\sup\scrT>\theta$. Fix a sufficiently large $k$ such that:
    \begin{equation}\label{eq:choice of k}
    t^*-\frac{2\pi b(k)+1}{d_k}>\theta;
    \end{equation}
    this is possible since $\displaystyle\lim_{k\to\infty} \frac{2\pi b(k)+1}{d_k}=0$.

    Since $t^*-\frac{\pi}{d_k}\in\scrT$, there is an open set $V$ containing $E([\theta,t^*-\frac{\pi}{d_k}],R)$ such that $F$ extends to an element of $\cF(D(0,R)\cup V)$. By compactness of $E([\theta,t^*-\frac{\pi}{d_k}],R)$ and cutting down $V$ if needed, we may assume that $V$ is the polar rectangle
    $$V=\left\{x_1e^{ix_2}:\ x_1\in (R-\delta,R+\delta),\ x_2\in \left(\theta-\delta,t^*-\frac{\pi}{d_k}+\delta\right)\right\}$$
    for a sufficiently small $\delta>0$. Let $V'$ be the ``further'' polar rectangle:
    $$V'=\left\{x_1e^{ix_2}:\ x_1\in(R-\delta,R+\delta),\ x_2\in \left(t^*-\frac{\pi}{d_k}-\delta,t^*+\frac{\pi}{d_k}+\delta\right)\right\}$$
    that has some small overlap with $V$.

    We now use the conclusion of Step 1 for $\cB_k$ and the $d_k$-th roots of unity
    $\zeta_1=\cdots=\zeta_{b(k)}=e^{2\pi i/d_k}$. For any subset $J$ of $\{1,\ldots,b(k)\}$, the notation $\zeta_J$ in Step 1 is simply:
    $$\zeta_J=e^{2\pi \vert J\vert i/d_k}.$$
    We have non-zero algebraic numbers $c$ and $c_{J}$ for $J\subsetneq \{1,\ldots,b(k)\}$ such that:
    \begin{equation}\label{eq:Step 2 key}
        cF(e^{2\pi b(k)i/d_k}x)+\sum_{J\subsetneq\{1,\ldots,b(k)\}} c_JF(e^{2\pi\vert J\vert  i/d_k}x)\in\Qbar(x).
    \end{equation}

    Let $y\in V'$, write $y=e^{2\pi b(k)i/d_k}x$. Then by our choice of 
    $k$, the sets $V$ and $V'$, every point $e^{2\pi \vert J\vert i/d_k}x$ with
    $J\subsetneq\{1,\ldots,b(k)\}$ is of the form $x_1e^{ix_2}$ with
    $x_1\in (R-\delta,R+\delta)$ and 
    $$t^*-\frac{\pi}{d_k}-\delta-\frac{2\pi b(k)}{d_k}<x_2<t^*-\frac{\pi}{d_k}+\delta.$$
    Combining this with \eqref{eq:choice of k}, we have:
    $$\theta-\delta<x_2<t^*-\frac{\pi}{d_k}+\delta,$$
    in other words every point $e^{2\pi \vert J\vert i/d_k}x$ with
    $J\subsetneq\{1,\ldots,b(k)\}$ is in $V$. Using \eqref{eq:Step 2 key} and the fact that $F$ can be extended to an element of $\cF(D(0,R)\cup V)$, we can conclude that $F$ can be extended to an element of $\cF(D(0,R)\cup V\cup V')$. This means $t^*+\frac{\pi}{d_k}\in \scrT$, contradicting the property $t^*=\sup \scrT$. Therefore $F$ can be extended to an element of $\cF(D(0,R'))$ for some $R'>R$.

    \textbf{Step 3:} we use Proposition~\ref{prop:KMN} to 
    finish the proof. By Step 2, we have $R'>R$ and  a rational function $F_1\in \Qbar(x)$ such that $F-F_1$ extends to an analytic function on $D(0,R')$. Since $F$ is analytic on $D(0,R)$, the poles of $F_1$ must be outside $D(0,R)$. By lowering $R'$ and removing the principal parts of $F_1$ at poles outside the closed disk $\overline{D(0,R)}$, we may assume that the poles of $F_1$, if any, lie on the circle of radius $R$. Write
    $$F_1(x)=\sum_{n=0}^{\infty}\tilde{a}_nx^n$$
    then for any $B'>1/R'$, we have
    $$\vert f(n)a_n-\tilde{a}_n\vert<B'^n$$
    for all sufficiently large $n$ since $F-F_1$ is the power series of an analytic function on $D(0,R')$.

    For all sufficiently large $n$, we can express
    $$\tilde{a}_n=\tilde{P}_1(n)\tilde{\alpha}_1^n+\cdots+\tilde{P}_{\tilde{r}}(n)\tilde{\alpha}_{\tilde{r}}^n$$
    where $\tilde{r}\in\N_0$, $\tilde{\alpha}_1,\ldots,\tilde{\alpha}_{\tilde{r}}$ are non-zero algebraic numbers on the circle of radius $R$, and 
    $\tilde{P}_1,\ldots,\tilde{P}_{\tilde{r}}\in \Qbar[x]\setminus\{0\}$. 
    We refer the reader to the notation at the beginning of the proof. We now fix $B'\in (1/R',1/R)$ sufficiently close to $1/R$
    and use the fact that $\vert f(n)\vert=e^{o(n)}$ and
    $\vert \alpha_i\vert<1/R$ for $s+1\leq i\leq r$ to conclude that:
    \begin{equation}\label{eq:inequality in n}
        \vert f(n)P_1(n)\alpha_1^n+\cdots+f(n)P_s(n)\alpha_s^n-(\tilde{P}_1(n)\tilde{\alpha}_1^n+\cdots+\tilde{P}_{\tilde{r}}(n)\tilde{\alpha}_{\tilde{r}}^n)\vert<B'^n
    \end{equation}
    for all large $n$.

    Let $\tilde{L}$ be the $\lcm$ of the orders of the roots of unity
    among the $\alpha_i/\alpha_j$'s, $\tilde{\alpha}_i/\tilde{\alpha}_j$'s, and
    $\alpha_{i}/\tilde{\alpha}_j$'s. We fix an arbitrary $b\in\{0,\ldots,\tilde{L}-1\}$ and restrict to the arithmetic progression $n=\tilde{L}m+b$. We can express:
    \begin{equation}\label{eq:rewrite a_n}
    P_1(n)\alpha_1^n+\cdots+P_s(n)\alpha_s^n=Q_1(m)\beta_1^m+\cdots+Q_t(m)\beta_t^m
    \end{equation}
    with $t\in\N$, non-degenerate tuple $(\beta_1,\ldots,\beta_t)$ with $\vert\beta_i\vert=1/R^{\tilde{L}}$, and
    non-zero $Q_i\in\Qbar[x]$ for $1\leq i\leq t$. The stability assumption on $\displaystyle\sum a_nx^n$ is used here to have $t\geq 1$, i.e. the RHS of \eqref{eq:rewrite a_n}
    is not the zero sequence.

    Similarly, we express:
    \begin{equation}\label{eq:rewrite tilde a_n}
    \tilde{P}_1(n)\tilde{\alpha}_1^n+\cdots+\tilde{P}_{\tilde{r}}(n)\tilde{\alpha}_s^n=\tilde{Q}_1(m)\tilde{\beta}_1^m+\cdots+\tilde{Q}_{\tilde{t}}(m)\tilde{\beta}_{\tilde{t}}^m
    \end{equation}
    with $\tilde{t}\in\N_0$, non-degenerate tuple $(\tilde{\beta}_1,\ldots,\tilde{\beta}_{\tilde{t}})$ with
    $\vert\tilde{\beta}_i\vert=1/R^{\tilde{L}}$, and
    non-zero $\tilde{Q}_i\in\Qbar[x]$ for $1\leq i\leq \tilde{t}$. Note that we allow the possibility
    $\tilde{t}=0$ here and the empty data $(\tilde{\beta}_i,\tilde{Q}_i)$ for $1\leq i\leq \tilde{t}$ simply mean that the RHS of \eqref{eq:rewrite tilde a_n}
    is the zero sequence in this case. 
    Moreover, by our choise of $\tilde{L}$, we have the further property that whenever $\beta_i/\tilde{\beta_j}$ is a root of unity then $\beta_i=\tilde{\beta}_j$.
    
    Then \eqref{eq:inequality in n} becomes:
    \begin{equation}\label{eq:inequality in m}
       \left\vert \sum_{i=1}^t f(\tilde{L}m+b)Q_i(m)\beta_i^m-\sum_{i=1}^{\tilde{t}} \tilde{Q}_i(m)\tilde{\beta}_i^m\right\vert<B'^{\tilde{L}m+b}
    \end{equation}
    for all sufficiently large $m$.
    There are 2 cases

    \textbf{Case 1:} $f(\tilde{L}m+b)=0$ for infinitely many $m$. Together with \eqref{eq:inequality in m}, we have:
    $$\left\vert \sum_{i=1}^{\tilde{t}}\tilde{Q}_i(m)\tilde{\beta}_i^m\right\vert<B'^{\tilde{L}m+b}$$
    for infinitely many $m$. Proposition~\ref{prop:KMN} implies that this happens only when $\tilde{t}=0$, in other words
    $\displaystyle\sum_{i=1}^{\tilde{t}}\tilde{Q}_i(m)\tilde{\beta}_i^m$ is the zero sequence. Combining this with \eqref{eq:inequality in m} implies that:
    $$\left\vert\sum_{i=1}^t f(\tilde{L}m+b)Q_i(m)\beta_i^m\right\vert<B'^{\tilde{L}m+b}$$
    for all sufficiently large $m$. Property (P2) guarantees that
    $h(f(\tilde{L}m+b)Q_i(m))=o(m)$ as $m\to\infty$. Proposition~\ref{prop:KMN}
    implies that $f(\tilde{L}m+b)=0$ for all sufficiently large $m$.

    \textbf{Case 2:} $f(\tilde{L}m+b)\neq 0$ for every sufficiently large $m$.
    By applying Proposition~\ref{prop:KMN} for \eqref{eq:inequality in m}, we must have that $\tilde{t}=t$ and there is a permutation $\sigma$ on $\{1,\ldots,t\}$
    such that
    $$\beta_i=\tilde{\beta}_{\sigma(i)}\quad\text{and}\quad f(\tilde{L}m+b)Q_i(m)=\tilde{Q}_{\sigma(i)}(m)$$
    for $1\leq i\leq t$ and every sufficiently large $m$. In particular, there is $A(x)\in\Qbar(x)$ such that $f(\tilde{L}m+b)=A(m)$ for all large $m$.

    It remains to show that $A$ is a polynomial using Property (P3) that $f$ is given by polynomials on most congruence classes modulo $d_k$.
    Let $d_k>\tilde{L}$ be large. There are at least $\lfloor d_k/\tilde{L}\rfloor$ many
    $j$'s in $\{0,\ldots,d_k-1\}$ such that $j\equiv b\bmod \tilde{L}$. 
    Therefore if
    $$\frac{\vert\cG_k\vert}{d_k}>1-\frac{1}{2\tilde{L}}$$
    then we can find $j\in \cG_k$ such that $j\equiv b\bmod \tilde{L}$. Write $j=b+q\tilde{L}$. Then for all sufficiently large $m$, we have:
    $$P_{k,j}(\tilde{L}m)=f(d_k\tilde{L}m+j)=f(d_k\tilde{L}m+q\tilde{L}+b)=A(d_km+q).$$
    Therefore $A\in\Qbar[x]$. This finishes the proof that $f$ is a quasi-polynomial.
    \end{proof}

	\begin{lemma}\label{lem:q-p and almost q-c}
	Let $f:\ \N_0\rightarrow \Qbar$ be both a quasi-polynomial and an almost quasi-constant. Then $f$ is a
	quasi-constant.
	\end{lemma}    
    \begin{proof}
    Since $f$ is a quasi-polynomial, there exist $\tilde{L}$ and $P_0,\ldots,P_{\tilde{L}-1}\in\Qbar[x]$ such that
    $f(n)=P_{i}(n)$ for all sufficiently large $n$ with $n\equiv i\bmod \tilde{L}$. 
    On the other hand, there exist $d_1<d_2<\ldots$ such that $f$ is given by a constant 
    polynomial on each of most of the congruence classes modulo $d_i$ as $i\to\infty$. By a completely analogous  argument to the proof that $A\in\Qbar[x]$ at the end of the proof of Theorem~\ref{thm:main}, we can conclude that each $P_{i}$ is constant.
    \end{proof}
    
    \begin{proof}[Proof of Theorem~\ref{thm:affirmative}]
      We assume the notation of Theorem~\ref{thm:affirmative} as in Section~\ref{sec:intro}. Note that $\varphi^k-1$ is not separable for some $k\geq 1$. Hence the field $K$ must have characteristic $p>0$. We have
      the following results from \cite[Section~2]{BC18_DO}:
      \begin{itemize}
          \item $N_k(\varphi)=\displaystyle\frac{\deg(\varphi^k-1)}{\deg_i(\varphi^k-1)}$ where $\deg_i$ denotes the inseparability degree.
          \item  $\deg(\varphi^k-1)=\prod_{i=1}^{2g}(\xi^k-1)$ where $\xi_1,\ldots,\xi_{2g}$ are algebraic numbers none of which is a root of unity.
          \item There exist sequences $(r_k)_{k\geq 1}$ and $(s_k)_{k\geq 1}$
          with $r_k\in\Q^*$, $s_k\in\Z_{\leq 0}$, and 
          $$\deg_i(\varphi^k-1)=r_k\vert k\vert_p^{s_k}$$
          for every $k$. Moreover, these sequences are periodic: there exists $L\in\N$ such that $r_{k+L}=r_k$ and $s_{k+L}=s_k$ for every $k$.
      \end{itemize}
      
      From \cite[Lemma~1]{BMW14_TA}, it suffices to prove that the power series
    $$\sum_{k=1}^{\infty}N_k(\varphi)x^k=\sum_{k=1}^{\infty}\frac{\prod_{i=1}^{2g}(\xi_i^k-1)}{r_k\vert n\vert^{s_k}}x^k$$
      admits the circle of radius $1/\Lambda$ as a natural boundary, where we recall that $\Lambda:=\displaystyle\prod_{i=1}^{2g}\max\{\vert\xi_i\vert,1\}$.
     
     We assume otherwise that the above power series can be extended analytically beyond the circle of radius $1/\Lambda$. Note that 
      the series $\displaystyle\sum_{k=1}^{\infty}\prod_{i=1}^{2g} (\xi_i^k-1)x^k$
      is stable, see Example~\ref{eg:non-degenerate is stable}.
      The function $f(k)=\displaystyle\frac{1}{r_k\vert n\vert_p^{s_k}}$
      is an almost quasi-constant (strictly speaking, this is defined for $k\in\N$, but we can assign any value for $f(0)$ to have $f$ on $\N_0$) by essentially the same reasoning in Example~\ref{eg:|n|_p}: take $d_i=Lp^i$, then
      $f$ is constant on each congruence class modulo $d_i$ except the congruence classes of $0,p^i,2p^i,\ldots,(L-1)p^i$ modulo $d_i$. Then Theorem~\ref{thm:main} implies that $f$ must be a quasi-polynomial. Lemma~\ref{lem:q-p and almost q-c} gives that $f$ is a quasi-constant.
    Therefore $f(k)=\displaystyle\frac{1}{\deg_i(\varphi^k-1)}$ only takes finitely many values for $k\in\N$. This contradicts \cite[Lemma~4.4]{BC18_DO}
    and we finish the proof.
    \end{proof}

    \section{Further applications}
    In this section, we present a few similar applications of Theorem~\ref{thm:main}. While the previous application, Theorem~\ref{thm:affirmative}, relies on the easy fact that the function $\vert n\vert_p$ is an almost quasi-constant, the key ingredient in the further applications is the more general result that a function of the form
    $$\vert P_1(n)\alpha_1^n+\cdots+P_r(n)\alpha_r^n\vert_v$$
    is an almost quasi-polynomial where $v$ is non-archimedean and the $\alpha_i$'s satisfy the normalized condition
    $\max_{i}\vert \alpha\vert_v=1$. In fact, it is more desirable to treat functions that are rational powers of the above. First, we have:
    \begin{lemma}\label{lem:Lm+b}
    Let $L\in\N$. A function $f:\ \N_0\rightarrow \Qbar$ is an almost quasi-polynomial (respectively almost quasi-constant) if and only if for each $b\in\{0,\ldots,L-1\}$ the function $f_b(m):=f(Lm+b)$ is so.
    \end{lemma}
    \begin{proof}
    We prove the lemma for almost quasi-polynomials, the case of almost quasi-constants is similar. For the ``only if'' direction, fix $b\in\{0,\ldots,L-1\}$ and it is immediate that $f_b$ satisfies (P1) and (P2). Let $d_1<d_2<\ldots$, the subset $\cG_i\subseteq\{0,\ldots,d_i-1\}$ for $i\geq 1$, and the  $P_{i,j}$'s be as in Definition~\ref{def:almost q-p}. Let 
    $\cB_i=\{0,\ldots,d_i-1\}\setminus\cG_i$ for $i\geq 1$.
    We consider the map $\Z/d_i\Z\rightarrow\Z/d_i\Z$ given by $m\mapsto Lm+b$. There are at most $L\vert \cB_i\vert$ many elements in the preimages
    of $\cB_i$ modulo $d_i$. Since $\vert\cB_i\vert=o(d_i)$ and outside those preimages the function $f_b$ is given by polynomials, we have that $f_b$ satisfies (P3). Therefore $f_b$ is an almost quasi-polynomial.

    We now prove the ``if'' direction. For each $b\in\{0,\ldots,L-1\}$, define:
    \[
    \tilde{f}_b(n)= 
    \begin{cases} 
    f(n) & \text{if } n \equiv b\bmod L,\\
    0       & \text{otherwise.}
    \end{cases}
    \]
    Since $f=\displaystyle\sum_{b=0}^{L-1}\tilde{f}_b$, it suffices to fix an arbitrary $b$ and prove that $\tilde{f}_b$ is an almost quasi-polynomial. Since $f_b$ is an almost quasi-polynomial, it follows immediately that $\tilde{f}_b$ satisfies (P1) and (P2). Let
    $d_1<d_2<\ldots$ and $\cG_i\subseteq\{0,\ldots,d_i-1\}$
    for each $i$ such that $f_b(m)$ is given by a polynomial for all sufficiently large $m$ in each congruence class $j$ modulo $d_i$ with $j\in\cG_i$.
    
    We consider the congruence classes modulo $Ld_1<Ld_2<\ldots$. Each congruence class modulo $Ld_i$ is of the form
    $$Lj+c\bmod Ld_i$$
    where $0\leq j\leq d_i-1$ and $0\leq c\leq L-1$. If $c\neq b$ then $\tilde{f}_b$ is identically zero on the above congruence class. If $c=b$ and $j\in \cG_i$ then $\tilde{f}_b$ is eventually given by a polynomial on the above congruence class. Therefore on each of at least
    $$(L-1)d_i+\vert\cG_i\vert$$
    many congruence classes modulo $Ld_i$, the function $\tilde{f}_b(n)$ is given by a polynomial when $n$ is large. Since $\displaystyle\lim_{i\to\infty}\frac{\vert\cG_i\vert}{d_i}=1$, we have:
    $$\lim_{i\to\infty}\frac{(L-1)d_i+\vert \cG_i\vert}{Ld_i}=1.$$
    Therefore $\tilde{f}_b$ satisfies (P3) and we finish the proof.
    \end{proof}

    \begin{lemma}\label{lem:p-adic analytic}
        Let $K$ be a number field and $v\in K_v$. Let $g(x)=\displaystyle\sum_{n=0}^{\infty} c_nx^n\in K_v[[x]]$
        that is convergent on $\{a\in K_v:\ \vert a\vert_v\leq 1\}$,
        equivalently $\displaystyle\lim_{n\to\infty} \vert c_n\vert_v=0$. Then the 
        function $f:\ \N_0\rightarrow\Qbar$ given by
        $f(n)=\vert g(n)\vert_v$ satisfies Property (P3) in which 
        the $P_{i,j}$'s are constant polynomials.
    \end{lemma}
    \begin{proof}
        The case $g=0$ is obvious, we assume otherwise. Let $p$ be the restriction of $v$ to $\Q$. We regard $g$ as an analytic
        function from $\Z_p$ to $K_v$. By Stra{\ss}mann's theorem \cite{Strassmann28}, $g$ has only finitely many zeros in $\Z_p$ denoted
        $z_1,\ldots,z_m$. 
        
        For $k\in\N$ that is sufficiently large, the elements
        $z_i\bmod p^k$ of $\Z_p/p^k\Z_p$ for $1\leq i\leq m$ are distinct. Let
        $\cC_k\subseteq\{0,1,\ldots,p^k-1\}$ such that $\cC_k$ mod $p^k$ is the complement of
        $\{z_i\bmod p^k:\ 1\leq i\leq m\}$ in $\Z_p/p^k\Z_p$. Let
        $$U_k=\bigcup_{c\in \cC_k} (c+p^k\Z_p)=\Z_p\setminus\bigcup_{i=1}^m (z_i+p^k\Z_p).$$
        By the definition and compactness of $U_k$, we have
        \begin{equation*}
            \ell_k:=\min\{\vert g(x)\vert_v:\ x\in U_k\}>0.
        \end{equation*}
        We note that $g$ is Lipschitz continuous: there exists $M>0$ such that
        $$\vert g(x)-g(y)\vert_v\leq M\vert x-y\vert_v$$
        for every $x,y\in\Z_p$. We choose a sufficiently large $t_k\in\N$ such that $\ell_k>M/p^{t_k}$. Then
        \begin{equation}\label{eq:g on Uk mod p^tk}
            \vert g(x+p^{t_k}y)-g(x)\vert_v<\vert g(x)\vert_v\quad\text{hence}\quad
            \vert g(x+p^{t_k}y)\vert_v=\vert g(x)\vert_v
        \end{equation}
        for every $x\in U_k$ and $y\in\Z_p$. By increasing $t_k$ if necessary, we may assume $t_k\geq k$.

        Put $d_k=p^{t_k}$ and let 
        \begin{align*}
        \cG_k&=\{0\leq j\leq d_k-1:\ j\equiv c\bmod p^k\ \text{for some $c\in\cC_k$}\}\\
        &=\{0\leq j\leq d_k-1:\ j\equiv u\bmod p^{t_k}\ \text{for some $u\in U_k$}\}.
        \end{align*}
        We have $\vert\cG_k\vert/d_k=(p^k-m)/p^k$, therefore $\lim \vert\cG_k\vert/d_k=1$. Moreover, \eqref{eq:g on Uk mod p^tk}
        gives that $f$ is constant on each congruence class $j\bmod d_k$ 
        for $j\in\cG_k$. This finishes the proof.
    \end{proof}

    \begin{theorem}\label{thm:SML variant}
    Let $K$ be a number field, $v\in M_K^0$, $c\in\Q$, and  $u_n=P_1(n)\alpha_1^n+\cdots+P_r(n)\alpha_r^n$ for $n\in\N_0$ a linear recurrence sequence defined over $K$. Suppose that $(u_n)_{n\geq 0}$ is $v$-stable and $\max\{\vert\alpha_i\vert_v:\ 1\leq i\leq r\}=1$. Then the function $f:\ \N_0\rightarrow \Qbar$ given by
    $f(n)=\vert u_n\vert_v^c$
    is an almost quasi-constant.
    \end{theorem}
    \begin{proof}
        For each $n$, we have $f(n)=p^{a_n}$ where $a_n\in\Q$ and its denominator is bounded in terms of $K$, $v$, and $c$. Hence 
        all the $f(n)$'s belong to a number field of the form $\Q(p^{1/D})$. This proves Property (P1). Let $(u_{v,n})_{n\geq 0}$ be the essential $v$-part of $(u_n)_{n\geq 0}$ as in Definition~\ref{def:v-stable}. For any $B\in (0,1)$, we have:
        $$B^n<\vert u_{v,n}\vert_v=\vert u_n\vert_v<(1/B)^n$$
        for all sufficiently large $n$. The upper bound is obvious, the lower bound
        follows from Corollary~\ref{cor:characterization}, and the middle equality
        follows from the lower bound and the ultrametric property of $v$. This implies the following:
        \begin{itemize}
            \item $a_n=o(n)$, hence $f$ satisfies Property (P2).
            \item We may work with $u_{v,n}$ instead of $u_n$. In other words we may assume $\vert\alpha_1\vert_v=\ldots=\vert\alpha_r\vert_v=1$.
        \end{itemize}

        By Skolem's method in the proof of the Skolem-Mahler-Lech theorem \cite{Sko34,Mah35,Lec53}, there exist $L\in\N$ and power series
        $g_0(x),\ldots,g_{L-1}(x)\in K_v[[x]]$ that are convergent on $\{a\in K_v:\ \vert a\vert_v\leq 1\}$ such that
        $u_{Lm+b}=g_b(m)$ for $0\leq b\leq L-1$ and $m\in \N_0$. By Lemma~\ref{lem:p-adic analytic}, the function $m\mapsto f(Lm+b)$ is an almost quasi-constant for $0\leq b\leq L-1$. We finish the proof by applying Lemma~\ref{lem:Lm+b}. 
    \end{proof}

    For the second application of Theorem~\ref{thm:main}, we have 
    the following result that implies a significant case of an open problem by Royals-Ward \cite[p.~2228]{BC18_DO} (see Remark~\ref{rem:RW}):
    \begin{corollary}\label{cor:RW}
    Let $(a_n)_{n\geq 0}$ and $(u_n)_{n\geq 0}$ be linear recurrence sequences
    defined over a number field $K$. Let $S$ be a finite subset of $M_K^0$
    and let $c=(c_v)_{v\in S}\in \Q^S$. Write
    $$\vert u_n\vert_{S,c}=\prod_{v\in S}\vert u_n\vert_v^{c_v}.$$
    Suppose that $(a_n)_{n\geq 0}$ is stable and
    $(u_n)_{n\geq 0}$ is $v$-stable for every $v\in S$.
    Then the power series $\displaystyle \sum_{n=0}^{\infty} \vert u_n\vert_{S,c}a_nx^n$ is either a rational function or admits the circle of radius of convergence as a natural boundary.
    \end{corollary}
    \begin{proof}
        Recall that our definition of linear recurrence sequences in this paper excludes the zero sequence. Let $R\in (0,\infty)$ be the radius of convergence of $\sum a_nx^n$. 
        Write $u_n=P_1(n)\alpha_1^n+\cdots+P_r(n)\alpha_r^n$ as before. For $v\in S$, let $M_v=\max_{1\leq i\leq r}\vert\alpha_i\vert_v$. Let $M=\prod_{v\in S} M_v^{c_v}$. By Theorem~\ref{thm:SML variant}, the function
        $n\mapsto \vert u_n\vert_v^{c_v}/M_v^{nc_v}$ is an almost quasi-constant.
        Therefore, the function
        $$n\mapsto \vert u_n\vert_{S,c}/M^n$$
        is an almost quasi-constant. By applying Theorem~\ref{thm:main} (and Remark~\ref{rem:radius}) to the power series $\sum a_nM^nx^n$, we have that
        $\sum \vert u_n\vert_{S,c}a_nx^n$ has radius of convergence $R/M$ and it is either a rational function or admits the circle of radius $R/M$ as a natural boundary.
    \end{proof}

    \begin{remark}\label{rem:RW}
    As explained in Example~\ref{eg:non-degenerate is stable}, many linear recurrence sequences including the non-degenerate ones and those of the form
    $\prod_{i=1}^m(\xi_i^n-1)$ are $v$-stable for every $v$. For such sequences $(u_n)_{n\geq 0}$, Corollary~\ref{cor:RW} 
    implies that 
    $\sum \vert u_n\vert_{S,c}u_nx^n$ satisfies the P\'olya-Carlson dichotomy
    for every finite set $S$ and tuple $c\in \Q^S$. This addresses the open problem by Royals-Ward \cite[p.~2228]{BC18_DO}.
    \end{remark}
    
    We conclude this paper with the solution of an open problem by Bell-Miles-Ward \cite[p.~664]{BMW14_TA}:
    \begin{theorem}\label{thm:BMW}
        Let $k\in\N$. For $1\leq j\leq k$, let $K_j$ be a number field, let $\xi_j\in K_j$ that is not a root of unity, and let $S_j$ be a finite (possibly empty) subset 
        of $M_{K_j}^0$. 
        For $1\leq j\leq k$ and $n\in\N$, put $c_{j,n}=\displaystyle\prod_{v\in M_{K_j}^{\infty}\cup S_j}\vert\xi_j^n-1\vert_v^{\delta_v}$.
        Then the power series
        $$Z(x):=\exp\left(\sum_{n=1}^{\infty} \prod_{j=1}^k c_{j,n}\frac{x^n}{n}\right)$$
        is either rational or has a natural boundary at its circle of convergence, and the latter occurs if and only if there exist $1\leq j\leq k$ and $v\in S_j$ such that
        $\vert\xi_j\vert_v=1$.
    \end{theorem}

    \begin{remark}\label{rem:BMW}
        This is \cite[Theorem~15]{BMW14_TA} without the extra condition that 
        $\vert \xi_j\vert_v\neq 1$ for every $1\leq j\leq k$ and $v\in M_{K_j}^{\infty}$. We remind the reader that the appearance of 
        $\delta_v:=[K_v:\Q_p]$ (where $v$ restricts to $p$ on $M_{\Q}$) in Theorem~\ref{thm:BMW} is due to the different normalization
        of $\vert\cdot\vert_v$ as explained in Remark~\ref{rem:normalization}. In
        \cite[p.~664]{BMW14_TA}, the authors remark that the above extra condition is essential in their proof of \cite[Theorem~15]{BMW14_TA} and they suggest the problem of removing this condition. Theorem~\ref{thm:BMW} solves this problem. The significance of the $c_{j,n}$'s and their product is that they are the number of periodic points of a certain automorphism on compact abelian groups, see
        \cite{CEW97,Mil08_PP,BMW14_TA}.
    \end{remark}

    \begin{proof}[Proof of Theorem~\ref{thm:BMW}]
        First, we prove the theorem
        when $Z(x)$ is replaced by
        $$F(x):=\sum_{n=1}^{\infty} \prod_{j=1}^k c_{j,n}x^n.$$
        Put $c_{j,n}^{\infty}=\displaystyle\prod_{v\in M_{K_j}^{\infty}} \vert\xi_j^n-1\vert_v^{\delta_v}$ and $c_{j,n}^0=\displaystyle\prod_{v\in S_j} \vert\xi_j^n-1\vert_v^{\delta_v}$.

        First, it is easy to show that $\displaystyle\sum_{n=1}^{\infty} \prod_{j=1}^k c_{j,n}^{\infty}x^n$ is a rational function. We include the proof for the convenience of the reader and for later use. We rewrite:
        \begin{equation}\label{eq:rewrite c_j,n,infty}
        c_{j,n}^{\infty}=\prod_{\sigma:\ K_j\rightarrow\C} \vert \sigma(\xi_j)^n-1\vert.
        \end{equation}
        For a pair of complex-conjugate embeddings $\sigma$ and $\bar{\sigma}$, we have:
        $$\vert \sigma(\xi_j)^n-1\vert\cdot\vert\bar{\sigma}(\xi_j)^n-1\vert=(\sigma(\xi_j)^n-1)\cdot(\bar{\sigma}(\xi_j)^n-1).$$
        For a real embedding $\sigma$, first consider the case $\sigma(\xi_j)>1$ and we obviously have:
        $$\vert\sigma(\xi_j)^n-1\vert=\sigma(\xi_j)^n-1.$$
        Then consider the case $\sigma(\xi_j)<-1$, we have:
        $$\vert\sigma(\xi_j)^n-1\vert = \begin{cases}
  1-\sigma(\xi_j)^n  & \text{if $n$ is odd} \\
  \sigma(\xi_j)^n-1  & \text{if $n$ is even}
\end{cases}$$
        There are similar expressions in the cases $\sigma(\xi_j)\in (-1,0)$ and
        $\sigma(\xi_j)\in (0,1)$. By using these expressions to expand the RHS of \eqref{eq:rewrite c_j,n,infty} then multiplying the $c_{j,n}^{\infty}$'s for $1\leq j\leq k$ and applying Proposition~\ref{prop:KMN}, we conclude that the power series
        $$\sum_{n=1}^{\infty}\prod_{j=1}^kc_{j,n}^{\infty}x^n$$
        represents a rational function with radius of convergence
        $$R=\left(\prod_{j=1}^k\prod_{v\in M_{K_j}^{\infty}}\max\{\vert\xi_j\vert_v^{\delta_v},1\}\right)^{-1},$$
        it is stable, and for $n\geq 1$ the coefficients $\displaystyle\prod_{j=1}^kc_{j,n}^{\infty}x^n$ 
        have the form:
        \begin{equation}\label{eq:the a_i's are integers}
		\prod_{j=1}^kc_{j,n}^{\infty}=a_1\alpha_1^n+\cdots+a_r\alpha_r^n,	        
        \end{equation}
        where the $\alpha_i$'s are distinct non-zero algebraic numbers and the 
        $a_i$'s are non-zero \emph{integers}.

        Now we proceed as in the proof of Corollary~\ref{cor:RW}. For $1\leq j\leq k$ and $v\in S_j$, put $M_{j,v}=\max\{\vert\xi_j\vert_v,1\}$.
        Then put $M=\displaystyle\prod_{j=1}^k\prod_{v\in S_j} M_{j,v}^{\delta_v}$.
        Theorem~\ref{thm:SML variant} implies that the function
        $n\mapsto \vert\xi_j^n-1\vert_v^{\delta_v}/M_{j,v}^{n\delta_v}$ is an almost quasi-constant. Therefore the function
        $$g(n):= \left(\prod_{j=1}^k c_{j,n}^0\right)/M^n$$
        is an almost quasi-constant. By applying Theorem~\ref{thm:main} (and Remark~\ref{rem:radius}) for this function and the power series
        $\displaystyle\sum_{n=1}^{\infty}\prod_{j=1}^kc_{j,n}^{\infty}M^nx^n$, we conclude that $F(x)$ has radius of convergence $R/M$ and it is either a rational function or admits the circle of radius $R/M$ as a natural boundary. 
        
        By the last assertion of Theorem~\ref{thm:main} and Lemma~\ref{lem:q-p and almost q-c}, the latter happens if and only if the above function $g(n)$ is
        a quasi-constant. If $\vert\xi_j\vert_v\neq 1$ for $1\leq j\leq k$ and $v\in S_j$ then $\vert\xi_j^n-1\vert_v=M_{j,v}^n$ for every $n$, hence $g\equiv 1$. On the other hand, suppose $\vert\xi_{\tilde{j}}\vert_{\tilde{v}}=1$ for some
        $1\leq\tilde{j}\leq k$ and $\tilde{v}\in S_{\tilde{j}}$. We have $\liminf \vert\xi_{\tilde{j}}^n-1\vert_{\tilde{v}}=0$ by 
        choosing $n$ such that $\xi_{\tilde{j}}^n-1$ is divisible by arbitrarily high power of the prime ideal corresponding $\tilde{v}$. And since 
        each function $n\mapsto \vert\xi_j^n-1\vert_v^{\delta_v}/M_{j,v}^{n\delta_v}$ is bounded from above, we have
        $\liminf g(n)=0$. Therefore the function $g$ cannot be a quasi-constant. This proves the conclusion of the theorem for $F(x)$. 

        From the definition of $Z(x)$ and $F(x)$ and the fact that $F(x)$ has radius of convergence $R/M$, we have that $Z(x)$ converges in the open disk $D(0,R/M)$. The relation $\displaystyle F(x)=\frac{xZ'(x)}{Z(x)}$
        implies that $F(x)$ can be extended to an analytic function on a connected open set strictly containing $D(0,R/M)$ \emph{if and only if} the same holds for $Z(x)$.
        When $F(x)$ admits the circle of radius $R/M$ as a natural boundary, the same holds for $Z(x)$ (and its radius of convergence must be $R/M$).
        When $F(x)$ represents a rational function then we have
        $\vert \xi_j\vert_v=1$ for every $1\leq j\leq k$ and $v\in S_j$, hence 
        $g\equiv 1$, and 
        $$F(x)=\sum_{n=1}^{\infty}\prod_{j=1}^kc_{j,n}^{\infty}M^nx^n$$
        where the coefficients $\displaystyle\prod_{j=1}^kc_{j,n}^{\infty}$ have the special form
        in \eqref{eq:the a_i's are integers}. We emphasize the fact that the $a_i$'s in \eqref{eq:the a_i's are integers} are integers. Then a standard 
        algebraic manipulation gives that $Z(x)$ is a rational function. This proves the conclusion for $Z(x)$.
    \end{proof}
    
	\bibliographystyle{amsalpha}
	\bibliography{Perturbation} 	
\end{document}